\documentclass[a4paper,reqno,fleqn]{amsart}

\usepackage{amsmath,amsthm,amssymb,graphicx,hyperref,pifont,txfonts}
\usepackage{esint}
\usepackage[usenames]{color}
\usepackage[a4paper,marginpar=25mm]{geometry} 

\theoremstyle{plain}

\theoremstyle{definition}

\theoremstyle{remark}

\numberwithin{equation}{section}

\begin{document}

\title[Non-planar interfaces]{Nucleation of austenite in mechanically stabilized martensite by localized heating}

\author{John M. Ball}
\address{John M. Ball: Mathematical Institute, University of Oxford, 24--29 St Giles', Oxford OX1 3LB, United Kingdom.}
\email{ball@maths.ox.ac.uk}

\author{Konstantinos Koumatos}
\address{Konstantinos Koumatos: Mathematical Institute, University of Oxford, 24--29 St Giles', Oxford OX1 3LB, United Kingdom.}
\email{koumatos@maths.ox.ac.uk}

\author{Hanu\v s Seiner}
\address{Hanu\v s Seiner: Institute of Thermomechanics ASCR, Dolej\v skova 5, 182 00 Prague 8, Czech Republic}
\email{hseiner@it.cas.cz}

\begin{abstract}
The nucleation of bcc austenite in a single crystal of a mechanically stabilized 2H-martensite of Cu-Al-Ni shape-memory alloy is studied. The nucleation process is induced by localized heating and observed by optical microscopy. It is observed that nucleation occurs after a time delay and that the nucleation points are always located at one of the corners of the sample (a rectangular bar in the austenite), regardless of where the localized heating is applied.

Using a simplified nonlinear elasticity model, we propose an explanation for the location of the nucleation points, by showing that the martensite is a local minimizer of the energy with respect to localized variations in the interior, on faces and edges of the sample, but not at some corners, where a localized microstructure can lower the energy.
\vspace{4pt}

\vspace{2pt}

\end{abstract}

\maketitle

\section{Introduction}
\label{sec:intro}

The shape-recovery process, i.e. the thermally driven transition from the low temperature phase (martensite) into the high-temperature phase (austenite), is a fundamental part of the shape-memory effect. For many shape-memory alloys, the critical temperature for initiation of the shape-recovery process is strongly dependent on the microstructure of martensite entering the transition. When the heating is applied on a thermally induced martensitic microstructure obtained by the stress-free cooling of the austenitic phase, the transition starts at a certain temperature, usually denoted as $A_S$ (austenite start).  However, if the material in the martensitic phase is, prior to the heating, deformed (i.e. if the microstructure is reoriented by application of external mechanical loads), this critical temperature can be shifted significantly upwards. This effect is called the mechanical stabilization of martensite and has been documented for both single crystals and polycrystalline shape-memory alloys (SMAs) \cite{stab_SC,stab_PC}.

The difference between the shape-recovery process from the mechanically stabilized martensite and from the thermally induced martensitic microstructure was clearly illustrated by acoustic emission (AE) measurements by Landa et al. \cite{ISPMA}. The AE method is based on detecting and counting the number of acoustic signals emitted by the material during the course of the transition (see \cite{Cernoch_JalCom, J_AE} for an example of the use of AE for characterization of the martensitic transitions in SMAs). Fig.~\ref{fig1_HS} (taken from \cite{ISPMA} with courtesy of M. Landa) gives an illustrative example of the comparison of AE records obtained for the same single crystal of the Cu-Al-Ni alloy undergoing the transition in these two different regimes. For the thermally induced microstructure, more than 90\% of AE events occur in a temperature range between the austenite start temperature $A_S$ and the austenite finish temperature $A_F$ , which is in agreement with DSC measurements for the same material\footnote{These temperatures, however, differ from the transition temperatures of the material used in the experimental section of this paper, since the heat treatment of the material used by Landa et al. \cite{ISPMA} was slightly different.}. The transition in this temperature interval is preceded by a small number of events (less than 10\%) appearing below $A_S$. These events can be ascribed to the formation of nuclei of austenite in the thermally induced martensitic microstructure. Above $A_S$, these nuclei grow successively through the material and provide the transition. For the stabilized martensite, more than 90\% of the events are recorded within a very narrow temperature interval. As observed by Seiner et al.~\cite{PhaseTran}, the transition from the mechanically stabilized martensite is provided by the formation and propagation of special interfacial microstructures, which interpolate between austenite and mechanically stabilized martensite ensuring the kinematically compatible connection between them. These microstructures are able to exist and propagate in a wide range of temperatures and thermal gradients \cite{PrEASc}. Thus, the AE record for the stabilized martensite can be interpreted as follows: the small number of AE events detected below the narrow interval corresponds to the nucleation of austenite. As soon as the nucleation barrier is overcome, the interfacial microstructure propagates abruptly through the specimen and no further increase of the temperature is necessary. This shows how essential the nucleation process is for the effect of mechanical stabilization and the shape-recovery process in general.

This mechanical stabilization effect resulted in a rather surprising nucleation mechanism of austenite in a Cu-Al-Ni single crystal. In a simplified setting, we provide a mathematical explanation for this mechanism, based on ideas of the modern calculus of variations.

\begin{figure}[ht]
\centering
\includegraphics[width=0.9\columnwidth]{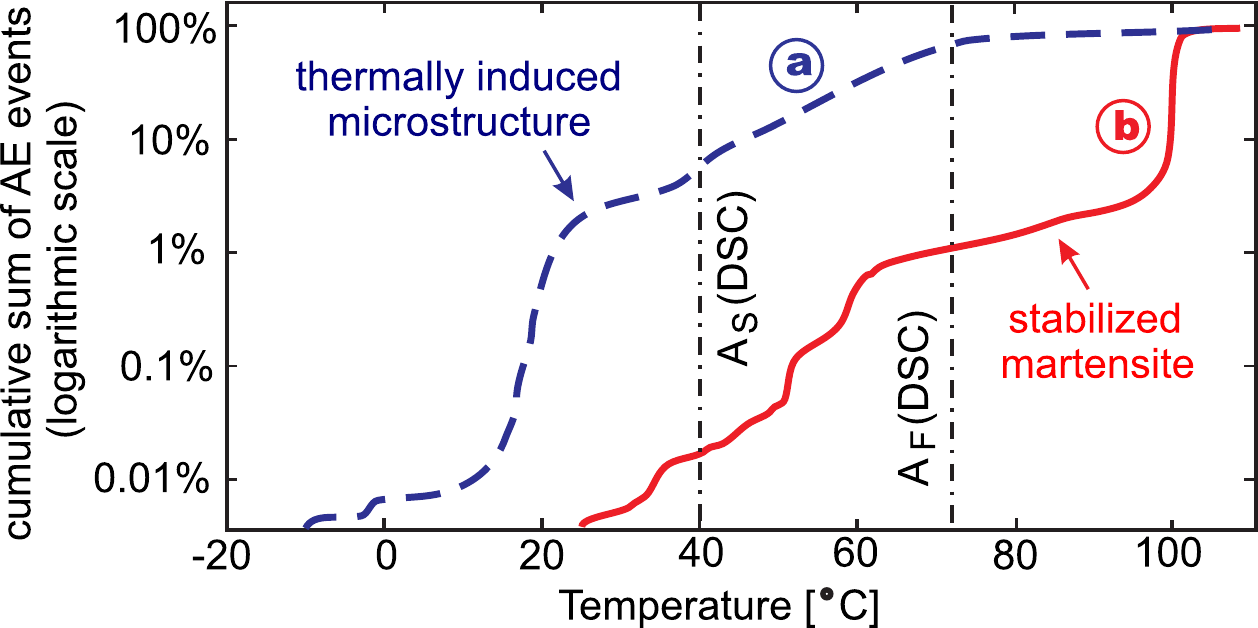}
\caption{Illustrative comparison of AE records for the transitions of Cu-Al-Ni single crystal from the thermally induced and mechanically stabilized states. (a) gradual increase of the number of events between $A_S$ and $A_F$ for the thermally induced microstructure; (b) abrupt transition of the stabilized martensite within a narrow temperature interval. The 100\% corresponds to $\sim$10$^7$ events.}
\label{fig1_HS}
\end{figure}

\section{Experimental observations}
\label{sec:experiment}

The observations that follow were made on a single crystal of Cu-Al-Ni, prepared by the Bridgeman method at the Institute of Physics, ASCR. The specimen was a prismatic bar of dimensions 12$\times$3$\times$3mm$^3$ in the austenite with edges approximately along the principal directions of the austenitic phase (see \cite{PhaseTran} for a detailed description). The martensite-to-austenite transition temperatures determined by DSC were $A_S=-6^\circ{\rm C}$ and $A_F=22^\circ{\rm C}$. The critical temperature $T_C$ for the transition from the stabilized martensite induced by homogeneous heating for this specimen was $\sim $60$^\circ$C. This was estimated from optical observations of the transition in this specimen with one of its faces laid on and thermally contacted with a gradually heated Peltier cell, using a heat conducting gel.

The specimen was subjected to the following experimental procedure: 
\begin{enumerate}
\item[a)] by unidirectional compression along its longest edge, the specimen was transformed into a single variant of mechanically stabilized 2H martensite. Due to the mechanical stabilization effect the reverse transition did not occur during unloading.
\item[b)]the specimen was then freely laid on a slightly pre-stressed, free-standing polyethylene (PE) foil (thickness 10$\mu$m, temperature resistance up to 140$^\circ$C). This ensured that there were minimal mechanical constraints to the specimen during the observations.
\item[c)]the specimen was locally heated by touching its surface with an ohmically heated tip of the Solomon SL-30 (Digital) soldering iron with temperature electronically controlled to be 200$^\circ$C (control accuracy $\sim \pm 5^\circ$C), i.e. significantly above the $A_S$ and $T_C$ temperatures. The nucleation of austenite was optically observed and recorded by a conventional CCD camera (7$\times$ optical zoom, 25 frames/second, PAL resolution with \verb#mpeg# compression).
\end{enumerate}

The localized heating was applied in three different ways: (i) with the tip touching one of the corners surrounding the upper face; (ii) with the tip touching one of the edges, approximately in the middle between two corners; (iii) with the tip touching approximately at the centre of the upper face.  These experiments were repeated for various orientations of the specimens, i.e. with various faces chosen to be the upper (observed) ones.

When heating was applied at a corner, the nucleation was always induced exactly at that corner and occurred nearly immediately after touching the specimen with the tip. When heating either an edge or the centre of the upper face, the nucleation occurred at one of the corners as well, i.e. the localized heating did not result in formation of the nucleus under the tip. Moreover, the nucleus was only observable after 30-60 s, which was enough time for the corner to reach the $T_C$ temperature. In different tests the nuclei were observed at different corners (including those lying on the PE foil) and the exact choice was probably governed by imperfections of the stabilized martensite. After the nucleation, the transition front formed and propagated through the specimen. The velocity of the transition front probably depended on the actual overheating of the specimen. For some runs of the experiment, it propagated at a few millimetres per second (comparable to the transition front propagating in a thermal gradient \cite{PrEASc}); for other runs, the whole specimen transformed fully within less than one second. This also supports the conjecture that the nucleation is affected by the local microstructure in the corners: if the nucleation barrier in one of the corners is lowered e.g. by imperfections in the stabilized martensite, the nucleation occurs earlier (i.e. at a lower temperature) and the transition front, which lowers the temperature of the material by the latent heat \cite{PrEASc}, propagates more slowly. 

In Fig.~\ref{fig3_HS}, snapshots from the observations are seen (\href{}{link to recorded video}). The transition fronts have morphologies of the interfacial microstructures described in \cite{PhaseTran} ($X-$ and $\lambda-$interfaces), in which the mechanically stabilized martensite is separated from austenite by a twinned region ensuring kinematical compatibility.

\begin{figure*}[ht]
\centering
\includegraphics[width=0.8\textwidth]{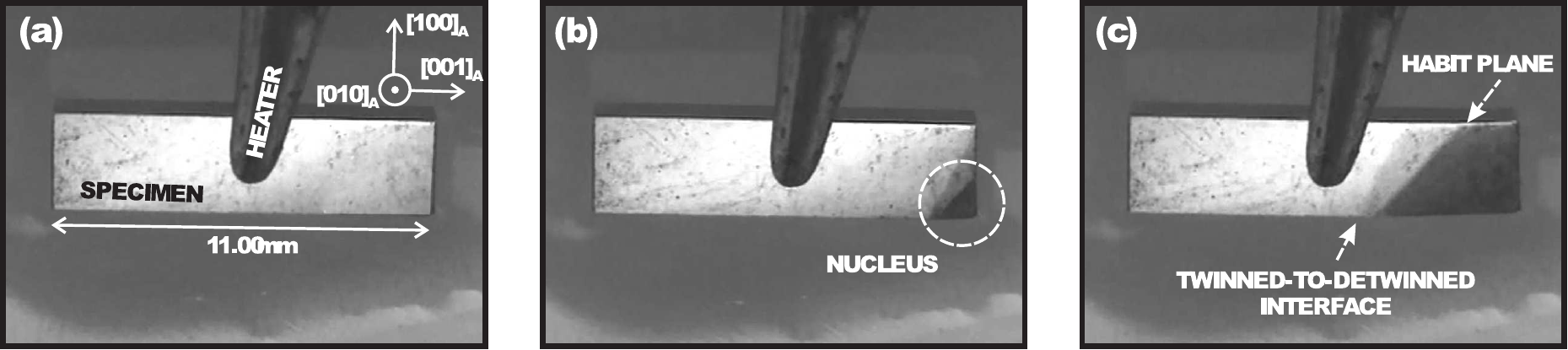}
\caption{Snapshots of the recorded video taken during the optical observations of the nucleation process. (a) the initial state with the length and crystallographic orientation of the specimen given in the coordinate system of the austenitic lattice (indicated by the subscript A); (b) formation of the nucleus at a corner (the first frame of the recorded video in which the nucleus was clearly visible); (c) the fully formed transition front propagating through the specimen. The morphology of the interfacial microstructure is outlined by the arrows indicating the austenite-to-twinned martensite interface (the habit plane) and the twinned-to-detwinned interface between the laminate and the stabilized martensite.}
\label{fig3_HS}
\end{figure*} 

\section{Nonlinear elasticity model: general and simplified}
\label{sec:model}

\subsection{General model}
\label{subsec:general}

The general nonlinear elasticity model~\cite{balljames87,balljames92}, which neglects interfacial energy, leads to the prediction of infinitely fine microstructures which are identified with limits of infimizing sequences $y^{k}$, $k=1,2,\ldots$, for a total free energy
\begin{equation}
E_{\theta}(y)=\int_{\Omega}\varphi(\nabla y(x),\theta)\,dx.\nonumber
\label{eq:E}
\end{equation}
Here, $\Omega$ represents the reference configuration of undistorted austenite at the critical temperature $\theta_{c}$ and $y(x)$ denotes the deformed position of the particle $x\in\Omega$. The free-energy function $\varphi(F,\theta)$ depends on the deformation gradient $F\in\,M^{3\times 3}$ and the temperature $\theta$ where $M^{3\times3}$ denotes the space of 3$\times$3 matrices. By frame indifference, $\varphi(RF,\theta)=\varphi(F,\theta)$ for all $F$, $\theta$ and for all rotations $R$; that is for all matrices in $SO(3)=\left\{R:R^{T}R=\mathbf{1}, \det{R}=1\right\}$. Let
\begin{equation}
K_{\theta}=\lbrace F:\varphi\left(G,\theta\right)\geq\varphi\left(F,\theta\right)\;\mathrm{for}\;\mathrm{all}\;\mathrm{matrices}\;G\rbrace\nonumber
\end{equation}
denote the set of energy-minimizing deformation gradients. Then we assume that
\[K_{\theta}=\left\{\begin{array}{ll}
\alpha\left(\theta\right)SO\left(3\right)\mbox{ - austenite}&\,\theta>\theta_{c}\\
SO\left(3\right)\cup\bigcup^{N}_{i=1}SO\left(3\right)U_{i}\left(\theta_c\right)&\,\theta=\theta_{c}\\
\bigcup^{N}_{i=1}SO\left(3\right)U_{i}\left(\theta\right)\mbox{ - martensite}&\,\theta<\theta_{c},
\end{array}\right.\]
where the positive definite, symmetric matrices $U_{i}\left(\theta\right)$ correspond to the $N$ distinct variants of martensite and $\alpha(\theta)$ is the thermal expansion coefficient of the austenite with $\alpha(\theta_{c})=1$.

However, information about the gradients of minimizing sequences $y^k$ for $E_{\theta}$ is lost in the limit $k\rightarrow\infty$ and a more convenient way to describe microstructure is via the use of gradient Young measures, which are families of probability measures $\nu=\left(\nu_{x}\right)_{x\in\Omega}$ generated by sequences of gradients $\nabla z^k$. Then we seek to minimize
\[I_{\theta}\left(\nu\right)=\int_{\Omega}\langle\nu_{x},\varphi\rangle\,dx=\int_{\Omega}\int_{M^{3\times3}}\varphi\left(A\right)\,d\nu_{x}\left(A\right)\]
over the space of gradient Young measures. In this case, the underlying (macroscopic) deformation gradient $\nabla z\left(x\right)$ corresponds to the centre of mass of $\nu$, $\nabla z\left(x\right)=\bar{\nu}_{x}=\langle\nu_{x},\mathrm{id}\rangle=\int_{M^{3\times3}}A\,d\nu_{x}\left(A\right)$ (see \cite{balljames92}). 

As an example of the use of Young measures, consider the $x$-independent measure $\nu_x=\lambda\delta_F+\left(1-\lambda\right)\delta_G$, for some $\lambda\in\left(0,1\right)$, supported on two rank-one connected matrices $F$ and $G=F+a\otimes n$ where $a$, $n$ are vectors and $\delta_{\cdot}$ denotes a Dirac mass. This Young measure is generated by gradients $\nabla z^k$ consisting of simple laminates formed from alternating layers with normal $n$ of width $\lambda k^{-1}$ and $\left(1-\lambda\right)k^{-1}$ in which $\nabla z^k$ takes the respective values $F$ and $G$ (see Fig.~\ref{fig:laminate}). At each $x$, $\nu_x$ gives the limiting probabilities $\lambda$, $1-\lambda$ as $k\rightarrow\infty$ of finding the matrices $F$ and $G$, respectively, in an infinitesimal neighbourhood of $x$. In this case, the macroscopic gradient is $\nabla z\left(x\right)=\bar{\nu}_{x}=\lambda F+\left(1-\lambda\right)G$.

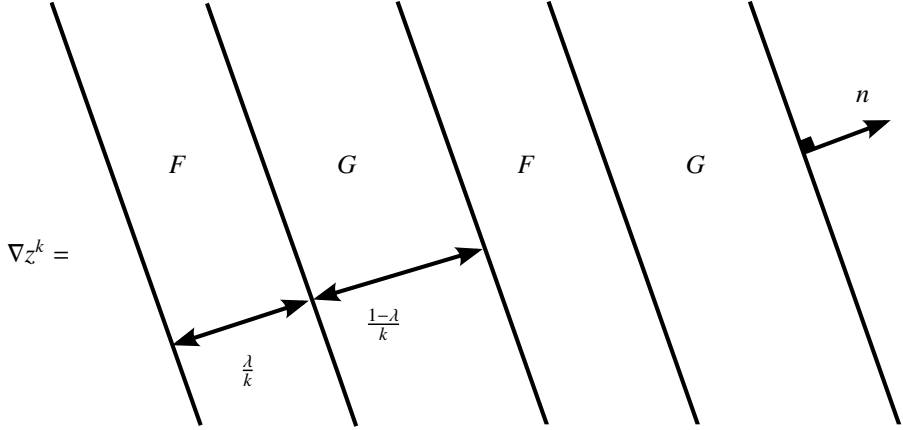
\begin{figure}[ht]
	\centering
	\def\svgwidth{0.8\columnwidth}
	\begingroup
    \setlength{\unitlength}{\svgwidth}
  \begin{picture}(1,0.5)%
    \put(0,0){\includegraphics[width=\unitlength]{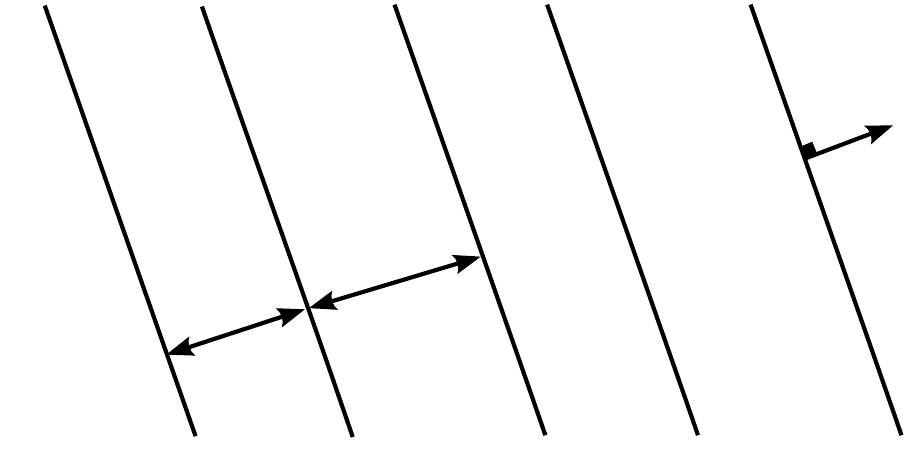}}%
    \put(0.18,0.3){\color[rgb]{0,0,0}\makebox(0,0)[lb]{\smash{$F$}}}%
    \put(0.37,0.3){\color[rgb]{0,0,0}\makebox(0,0)[lb]{\smash{$G$}}}%
    \put(0.57,0.3){\color[rgb]{0,0,0}\makebox(0,0)[lb]{\smash{$F$}}}%
    \put(0.76,0.3){\color[rgb]{0,0,0}\makebox(0,0)[lb]{\smash{$G$}}}%
    \put(0.95,0.38){\color[rgb]{0,0,0}\makebox(0,0)[lb]{\smash{$n$}}}%
    \put(0.26,0.07){\color[rgb]{0,0,0}\makebox(0,0)[lb]{\smash{$\frac{\lambda}{k}$}}}%
    \put(0.4,0.122){\color[rgb]{0,0,0}\makebox(0,0)[lb]{\smash{$\frac{1-\lambda}{k}$}}}%
    \put(0,0.2){\color[rgb]{0,0,0}\makebox(0,0)[lb]{\smash{$\nabla z^k=$}}}%
  \end{picture}%
\endgroup
	\caption{Sequence of gradients $\nabla z^{k}$ generating the $x$-independent Young measure $\nu_{x}=\lambda\delta_F+\left(1-\lambda\right)\delta_G$.}
	\label{fig:laminate}
\end{figure}

\subsection{Simplified model}
\label{subsec:simplified}

For our simplified model, we assume that $\theta>\theta_{c}$ and drop the explicit dependence on the temperature. Let $\Omega$ denote the Cu-Al-Ni bar in the austenite at $\theta=\theta_{c}$ and $\varphi:M^{3\times 3}_{+}\longrightarrow\mathbb{R}\cup\lbrace+\infty\rbrace$\footnote{$M^{3\times3}_{+}$ denotes the space of 3 by 3 matrices with positive determinant.} be the free-energy function for the material. Since $\theta>\theta_{c}$, we may assume that $\varphi$ is bounded below by some $-\delta<0$ and that
\begin{equation}
\varphi\left(F\right)=\left\{\begin{array}{ll}
-\delta&\,F\in\,SO\left(3\right)\\
0&\,F\in\bigcup^{6}_{i=1}SO\left(3\right)U_{i},
\end{array}\right.
\label{eq:phi}
\end{equation}
where the matrices $U_{i}$ correspond to the six martensitic variants for the cubic-to-orthorhombic transition of Cu-Al-Ni given by
\begin{equation}
\begin{array}{cc}
U_{1}=\left(\begin{array}{ccc}
\beta & 0 & 0\\0 & \frac{\alpha+\gamma}{2} & \frac{\alpha-\gamma}{2}\\0 & \frac{\alpha-\gamma}{2} & \frac{\alpha+\gamma}{2}\end{array}\right) & U_{2}=\left(\begin{array}{ccc}
\beta & 0 & 0\\0 & \frac{\alpha+\gamma}{2} & \frac{\gamma-\alpha}{2}\\0 & \frac{\gamma-\alpha}{2} & \frac{\alpha+\gamma}{2}\end{array}\right)
\end{array}\nonumber
\end{equation}
\begin{equation}
\begin{array}{cc}
U_{3}=\left(\begin{array}{ccc}
\frac{\alpha+\gamma}{2} & 0 & \frac{\alpha-\gamma}{2}\\0 & \beta & 0\\\frac{\alpha-\gamma}{2} & 0 & \frac{\alpha+\gamma}{2}\end{array}\right) & U_{4}=\left(\begin{array}{ccc}
\frac{\alpha+\gamma}{2} & 0 & \frac{\gamma-\alpha}{2}\\0 & \beta & 0\\\frac{\gamma-\alpha}{2} & 0 & \frac{\alpha+\gamma}{2}\end{array}\right)
\end{array}\nonumber
\end{equation}
\begin{equation}
\begin{array}{cc}
U_{5}=\left(\begin{array}{ccc}
\frac{\alpha+\gamma}{2} & \frac{\alpha-\gamma}{2} & 0\\\frac{\alpha-\gamma}{2} & \frac{\alpha+\gamma}{2} & 0\\0 & 0 & \beta\end{array}\right) &	U_{6}=\left(\begin{array}{ccc}
\frac{\alpha+\gamma}{2} & \frac{\gamma-\alpha}{2} & 0\\\frac{\gamma-\alpha}{2} & \frac{\alpha+\gamma}{2} & 0\\0 & 0 & \beta\end{array}\right).
\end{array}\nonumber
\end{equation}

In order to make the problem more tractable we work with an energy functional that captures the essential behaviour of $\varphi$ but becomes infinite off the energy wells
\[K:=SO\left(3\right)\cup\bigcup^{6}_{i=1}SO\left(3\right)U_{i}.\]
In particular, we employ $\Gamma$-convergence to rigorously derive this functional (see~\cite{ballkoumatos} for details). For $k=1,2,\ldots,$ let $\varphi^{k}=k\psi+\varphi$ where $\psi:M^{3\times3}\longrightarrow\mathbb{R}$ is a map such that $\psi\geq0$ and $\psi\left(A\right)=0$ if and only if $A\in\,K$. For a Young measure $\nu=\left(\nu_x\right)_{x\in\Omega}$ and eack $k=1,2,\ldots,$ define the energies $I^{k}\left(\nu\right)=\int_{\Omega}\langle\nu_x,\varphi^{k}\rangle\,dx$.

The idea behind $\Gamma$-convergence is to precisely introduce a suitable notion of `variational convergence' for which whenever $I^k$ $\Gamma$-converges to $I$ then $\min I=\lim_{k\rightarrow\infty}\inf I^{k}$ and if $\nu^{k}$ is a converging sequence such that $\lim_{k}I^{k}\left(\nu^{k}\right)=\lim_{k}\inf I^{k}$, then its limit is a minimum point for $I$; here, infima and minima are taken over the space of Young measures. In our case, one expects that as $k\rightarrow\infty$ the increasing term $k\psi$ will force the limiting energy to blow up everywhere outside $K$. Indeed, one can show that $I^{k}$ $\Gamma$-converges to

\begin{equation}
I\left(\nu\right)=\int_{\Omega}\langle\nu_{x},W\rangle\,dx=\int_{\Omega}\int_{M^{3\times3}}W\left(A\right)\,d\nu_{x}\left(A\right)dx,
\end{equation}
where $W\left(A\right)=\varphi\left(A\right)$ for all $A\in\,K$ and $W\left(A\right)=+\infty$ otherwise. Note that this energy forces minimizers to be supported entirely within the set $K$.

\section{Why nucleation can only occur at a corner}
\label{sec:result}

Let $U_{s}$ be the stabilized variant of martensite so that $\delta_{U_{s}}$ is the Young measure corresponding to a pure phase of that variant. In our minimization problem, we consider variations of $\delta_{U_{s}}$ which are localized in the interior, on faces, edges and at corners. More precisely, letting $B_{i},\,B_{f},\,B_{e},\,B_{c}$ be as in Fig.~\ref{fig:admissible}, we say that a measure $\nu=\left(\nu_x\right)_{x\in\Omega}$ is \textit{admissible} for the interior (resp. for a face, an edge, a corner) if $\nu_{x}=\delta_{U_{s}}$ outside $B_{i}$ (resp. $B_{f}$, $B_{e}$, $B_{c}$) and $\bar{\nu}_{x}=\nabla y\left(x\right)$ almost everywhere in $\Omega$ for some $y$ with $y\left(x\right)=U_{s}x$ on the boundary $\partial B_{i}$ of $B_{i}$ (resp. $\partial B_{f}\cap\Omega$, $\partial B_{e}\cap\Omega$, $\partial B_{c}\cap\Omega$)\footnote{Technically, $\nu$ is required to be a $W^{1,\infty}$ gradient Young measure meaning that it is generated by a sequence of gradients $\nabla z^k$ such that for some $M$, $\vert\nabla z^k\left(x\right)\vert\leq M<\infty$ for all $k$ and a.e. $x$; then the corresponding `weak limit' $z$ of $z^k$ also satisfies $\vert\nabla z\left(x\right)\vert\leq M$.}. For faces, edges and corners $\partial B_{f}\cap\partial\Omega$, $\partial B_{e}\cap\partial\Omega$ and $\partial B_{c}\cap\partial\Omega$ act as free boundaries.

\begin{figure}[ht]
	\centering
	\def\svgwidth{0.9\columnwidth}
	\begingroup
    \setlength{\unitlength}{\svgwidth}
  \begin{picture}(1,0.55363752)%
    \put(0,0){\includegraphics[width=\unitlength]{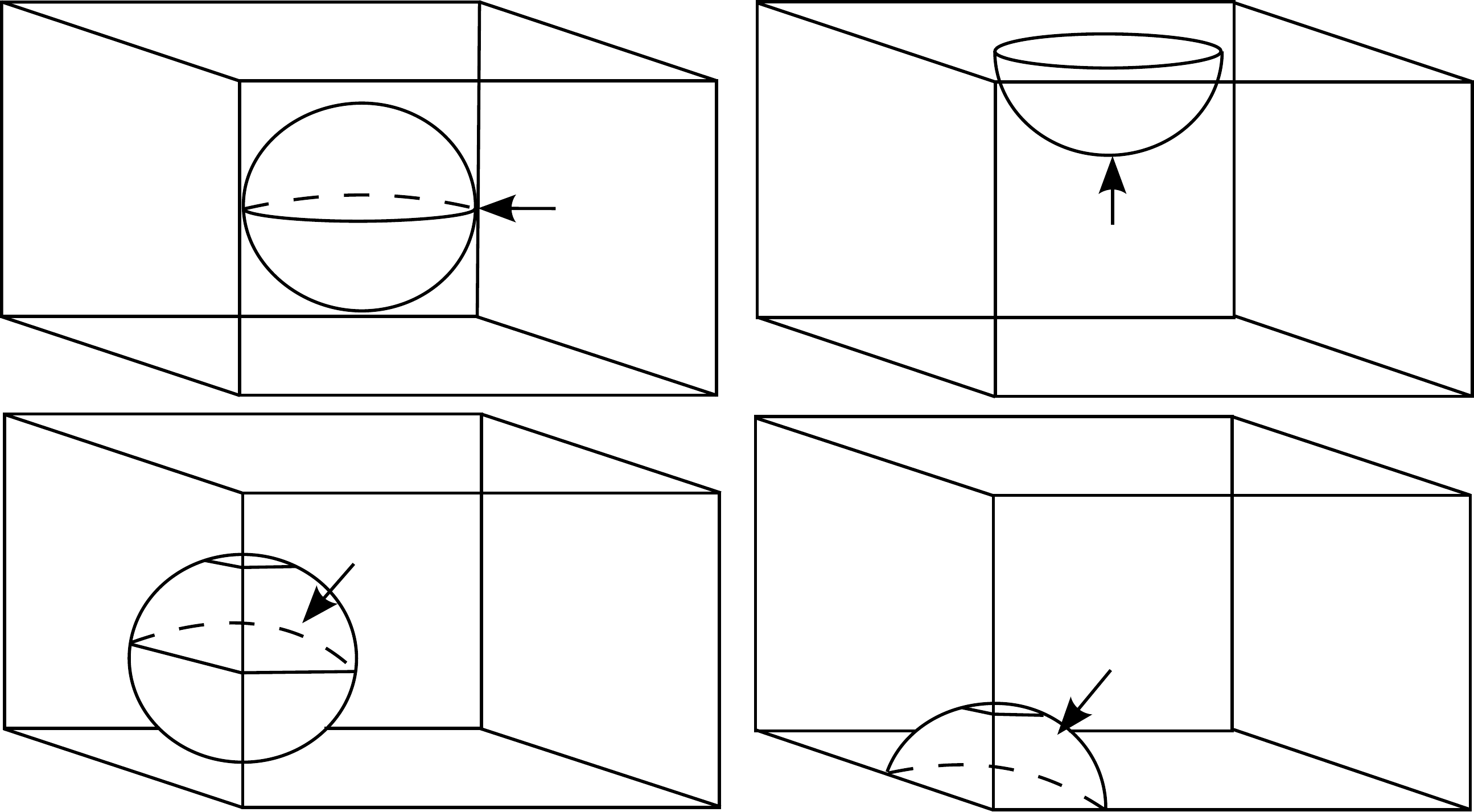}}%
    \put(0.38,0.4){\color[rgb]{0,0,0}\makebox(0,0)[lb]{\smash{$B_i$}}}%
    \put(0.737,0.365){\color[rgb]{0,0,0}\makebox(0,0)[lb]{\smash{$B_f$}}}%
    \put(0.24,0.17){\color[rgb]{0,0,0}\makebox(0,0)[lb]{\smash{$B_e$}}}%
    \put(0.74,0.1){\color[rgb]{0,0,0}\makebox(0,0)[lb]{\smash{$B_c$}}}%
    \put(0.014,0.46){\color[rgb]{0,0,0}\makebox(0,0)[lb]{\smash{interior}}}%
    \put(0.545,0.46){\color[rgb]{0,0,0}\makebox(0,0)[lb]{\smash{face}}}%
    \put(0.017,0.18){\color[rgb]{0,0,0}\makebox(0,0)[lb]{\smash{edge}}}%
    \put(0.54,0.18){\color[rgb]{0,0,0}\makebox(0,0)[lb]{\smash{corner}}}%
  \end{picture}%
\endgroup
	\caption{Subsets of $\Omega$ used for testing whether nucleation of austenite can occur in the interior, on a face, an edge and at a corner; these are given respectively by the intersection of $\Omega$ with a small ball centred at a point in the interior, on a face, an edge or a corner.}
	\label{fig:admissible}
\end{figure}
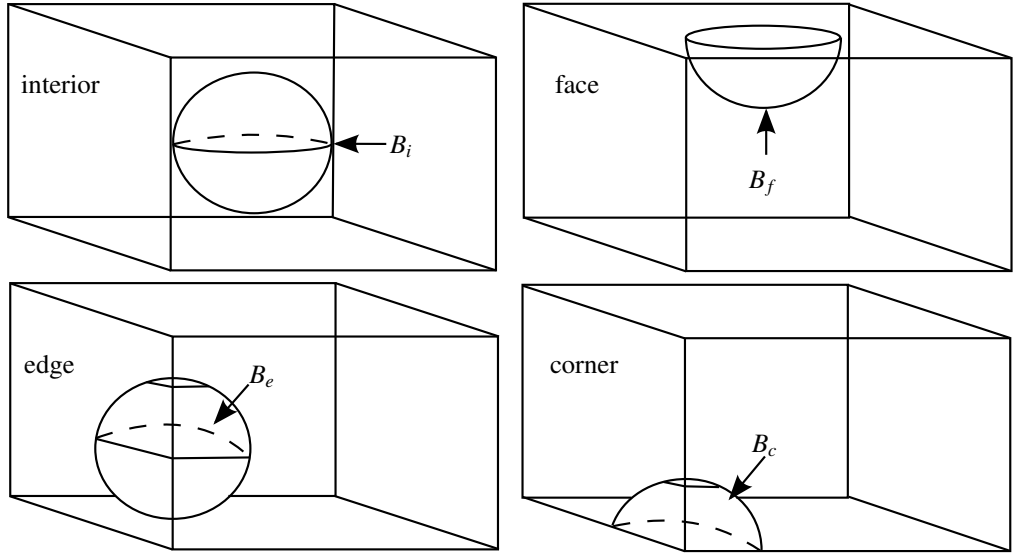

We also assume that $\det\,U_{s}\leq1$ and that
\begin{equation}
\int_{\Omega}\det\nabla y\left(x\right)\,dx\leq\mathrm{vol}\left(y\left(\Omega\right)\right)
\label{eq:cn}
\end{equation}
for any map $y$ underlying an admissible measure $\nu$, i.e. $\nabla y\left(x\right)=\bar{\nu}_{x}$. Condition (\ref{eq:cn}) was introduced by Ciarlet and Ne\v{c}as~\cite{cn} as a way to describe non-interpenetration of matter. We denote the sets of admissible measures $\nu=\left(\nu_x\right)_{x\in\Omega}$ for the interior, faces, edges and corners by $\mathcal{A}_{i}$, $\mathcal{A}_{f}$, $\mathcal{A}_{e}$ and $\mathcal{A}_{c}$ respectively.

For $s=1,\ldots\,,6$ and $S^{2}=\lbrace e\in\mathbb{R}^{3}:\vert e\vert =1\rbrace$, the unit sphere, let
\begin{eqnarray}
\mathcal{M}_{s}&=&\lbrace e\in S^{2}:\vert U_{s}e\vert=\max_{i}\lbrace\vert U_{i}e\vert ,1\rbrace\rbrace\quad\mbox{and}\nonumber\\
\mathcal{M}^{-1}_{s}&=&\lbrace e\in S^{2}:\vert\mathrm{cof}\,U_{s}e\vert > \max_{i\neq s}\lbrace\vert\mathrm{cof}\,U_{i}e\vert ,1\rbrace\rbrace\cup\lbrace e_{\max}(\mathrm{cof}\,U_s)\rbrace,\nonumber
\end{eqnarray}
where, for $F\in\,M^{3\times3}$, $\mathrm{cof}\,F$ stands for the matrix of all 2$\times$2 subdeterminants of $F$, $e_{\max}(F)$ stands for the eigenvector of $F$ corresponding to its largest eigenvalue and $\vert F\vert=\sqrt{\mathrm{Tr}\,F^{T}F}$ denotes the Euclidean norm in $M^{3\times3}$.
\newtheorem{thm}{Theorem}
\begin{thm}
\cite{ballkoumatos} Let $\Omega$ be a parallelepiped (not necessarily rectangular) with edges in the direction of vectors in $\mathcal{M}_{s}\cup U^{-2}_{s}\mathcal{M}^{-1}_{s}$. Assume that there exists a Young measure $\nu\in\displaystyle\mathcal{A}_{i}\cup\mathcal{A}_{f}\cup\mathcal{A}_{e}\cup\mathcal{A}_{c}$ such that $I\left(\nu\right)<I\left(\delta_{U_{s}}\right)$. Then, $\nu\in\mathcal{A}_{c}$.
\label{thm:main}
\end{thm}

\begin{proof}[Proof (sketch)]
Let $\Omega$ be as in the statement and let $\nu=\left(\nu_x\right)_{x\in\Omega}$ be an element of $\mathcal{A}_{i}\cup\mathcal{A}_{f}\cup\mathcal{A}_{e}\cup\mathcal{A}_{c}$ such that $I\left(\nu\right)<I\left(\delta_{U_{s}}\right)$. We first show that $\nu\notin\mathcal{A}_{i}$. Note that since $I\left(\delta_{U_{s}}\right)=0$ we may assume that $\mathrm{supp}\,\nu_{x}\subset K$ as otherwise $I\left(\nu\right)=+\infty$ and the result is trivial. By averaging the measure $\nu$ (see \cite{ballkoumatos}) we may also assume that $\nu$ is an $x$-independent Young measure and $\bar{\nu}=U_{s}$ without altering the energy $I\left(\nu\right)$. The minors relation for the determinant (see e.g. \cite{balljames92}, \cite{bhattacharya}) says that $\det\bar{\nu}=\langle\nu,\det\rangle$ and hence,
\begin{equation}
\det U_{s}=\int_{SO\left(3\right)}\det A\;d\nu\left(A\right)+\int_{\bigcup_{i}SO\left(3\right)U_{i}}\det A\;d\nu\left(A\right)\nonumber
\end{equation}
\begin{equation}
=\int_{SO\left(3\right)}1\,d\nu\left(A\right)+\int_{\bigcup_{i}SO\left(3\right)U_{i}}\det U_{s}\;d\nu\left(A\right)
\label{eq:proof1}
\end{equation}
since $\det U_l=\det U_s$ for all $l$. Also, $\nu$ is a probability measure, i.e. $\int_{K}\,d\nu\left(A\right)=1$, so that
\begin{equation}
\det U_{s}=\int_{SO\left(3\right)}\det U_{s}\;d\nu\left(A\right)+\int_{\bigcup_{i}SO\left(3\right)U_{i}}\det U_{s}\;d\nu\left(A\right)\nonumber
\end{equation}
and subtracting from (\ref{eq:proof1}), \[\int_{SO\left(3\right)}\left(1-\det U_{s}\right)\;d\nu\left(A\right)=0.\]

Hence, $\nu\left(SO\left(3\right)\right)=\int_{SO\left(3\right)}\,d\nu\left(A\right)=0$ or $\det\,U_{s}=1$. The former case leads to a contradiction as then \[I\left(\nu\right)=\int_{\Omega}\int_{\bigcup_{i}SO\left(3\right)U_i}W\left(A\right)d\nu\left(A\right)dx=0=I\left(\delta_{U_{s}}\right).\]
So, let $\det\,U_{s}=\alpha\beta\gamma=1$. By the AM-GM inequality
\begin{equation}
\frac{\vert U_{s}\vert^{2}}{3}=\frac{\alpha^{2}+\beta^{2}+\gamma^{2}}{3}\geq\left(\alpha^{2}\beta^{2}\gamma^{2}\right)^{1/3}=1\nonumber
\end{equation}
and thus $\vert U_{s}\vert^{2}>3=\vert\mathbf{1}\vert^{2}$. Note that the inequality is strict as otherwise $\alpha=\beta=\gamma=1$ and $U_{i}=\mathbf{1}$ for all $i=1,\ldots\,,6$. The map $F\mapsto \vert F\vert^{2}$ is convex and so $\vert\bar{\nu}\vert^2\leq\langle\nu,\vert\cdot\vert^2\rangle$. Then
\begin{equation}
\vert U_{s}\vert ^2\leq\int_{SO\left(3\right)}\vert A\vert^{2}\;d\nu\left(A\right)+\int_{\bigcup_{i}SO\left(3\right)U_{i}}\vert A\vert^{2}\;d\nu\left(A\right)\nonumber
\end{equation}
\begin{equation}
=\int_{SO\left(3\right)}3\;d\nu\left(A\right)+\int_{\bigcup_{i}SO\left(3\right)U_{i}}\vert U_{s}\vert^{2}\;d\nu\left(A\right)
\label{eq:proof3}
\end{equation}
since the norm does not change on martensitic variants. As $\nu$ is a probability measure,
\[\vert U_{s}\vert^{2}=\int_{SO\left(3\right)}\vert U_{s}\vert^{2}\;d\nu\left(A\right)+\int_{\bigcup_{i}SO\left(3\right)U_{i}}\vert U_{s}\vert^{2}\;d\nu\left(A\right)\]
and subtracting from (\ref{eq:proof3}), \[\int_{SO\left(3\right)}\left(\vert U_{s}\vert^{2}-3\right)\;d\nu\left(A\right)\leq0.\] However, $\vert U_{s}\vert^{2}>3$ and hence, $\nu\left(SO\left(3\right)\right)=0$ completing the case of the interior. Note that the proof does not utilize (\ref{eq:cn}) or the condition that $\det\,U_{s}\leq1$; these are only relevant for faces and edges. Also, the result for the interior does not dependent on the orientation of $\Omega$.

As for faces or edges, we wish to deduce that $\nu$ cannot be an element of $\mathcal{A}_{f}$ or $\mathcal{A}_{e}$. The proofs, though similar, are more involved and we refer the reader to \cite{ballkoumatos} for details. The proofs essentially rely on showing that whenever a line segment joins points on the prescribed part of the boundary $\partial B_{f}\cap\Omega$ or $\partial B_{e}\cap\Omega$ of $B_{f}$ or $B_{e}$, respectively, and lies in the direction of a vector in $\mathcal{M}_{s}\cup U^{-2}_{s}\mathcal{M}^{-1}_{s}$, then it must necessarily deform like $U_{s}x$ under any map $y$ underlying an admissible measure $\nu\in\mathcal{A}_{f}$ or $\mathcal{A}_{e}$. 

If the normal to a face is perpendicular to, or an edge is in the direction of, a vector in $\mathcal{M}_{s}\cup U^{-2}_{s}\mathcal{M}^{-1}_{s}$, the sets $B_{f}$ or $B_{e}$ can then be covered by such line segments so that $y\left(x\right)=U_{s}x$ in $\Omega$. But this means that $\bar{\nu}_{x}=U_{s}$ and in a manner very similar to the proof for the interior, we can show that this implies $I\left(\nu\right)=0$, i.e. for all $\nu\in\mathcal{A}_{f}$ or $\mathcal{A}_{e}$, $I\left(\nu\right)\geq I\left(\delta_{U_{s}}\right)$ and no admissible measure for a face or edge can lower the energy.
\end{proof}

On the other hand, a specific construction shows that for the Cu-Al-Ni specimen of this paper and some corners (see \cite{ballkoumatos} for details) there exists a measure $\nu\in\mathcal{A}_{c}$ such that $I\left(\nu\right)<I\left(\delta_{U_{s}}\right)$. In this construction (see Fig.~\ref{fig:min_corner}) the measure $\nu$ takes the value $\delta_R$ in a small region at a corner, for some $R\in\,SO\left(3\right)$. The rotation $R$ can itself form a compatible interface with a simple laminate as in Fig.~\ref{fig:laminate} with $F=U_{s}$ and $G=QU_{l}$ for some variant chosen to form the interface with $R$. This laminate can trivially also form a compatible interface with a pure phase of the variant $U_{s}$ and serves as the interfacial microstructure interpolating between $R$ (austenite) and $U_{s}$ making the entire microstructure compatible. Note that since the measure $\nu$ is supported on $SO\left(3\right)$ it must indeed lower the energy. Then, Theorem~\ref{thm:main} combined with the existence of an admissible measure in $\mathcal{A}_{c}$ that lowers the energy imply that nucleation must, and does, occur at a corner.

\begin{figure}[ht]
	\centering
	\def\svgwidth{0.8\columnwidth}
	\begingroup
    \setlength{\unitlength}{\svgwidth}
  \begin{picture}(1,0.59959741)%
    \put(0,0){\includegraphics[width=\unitlength]{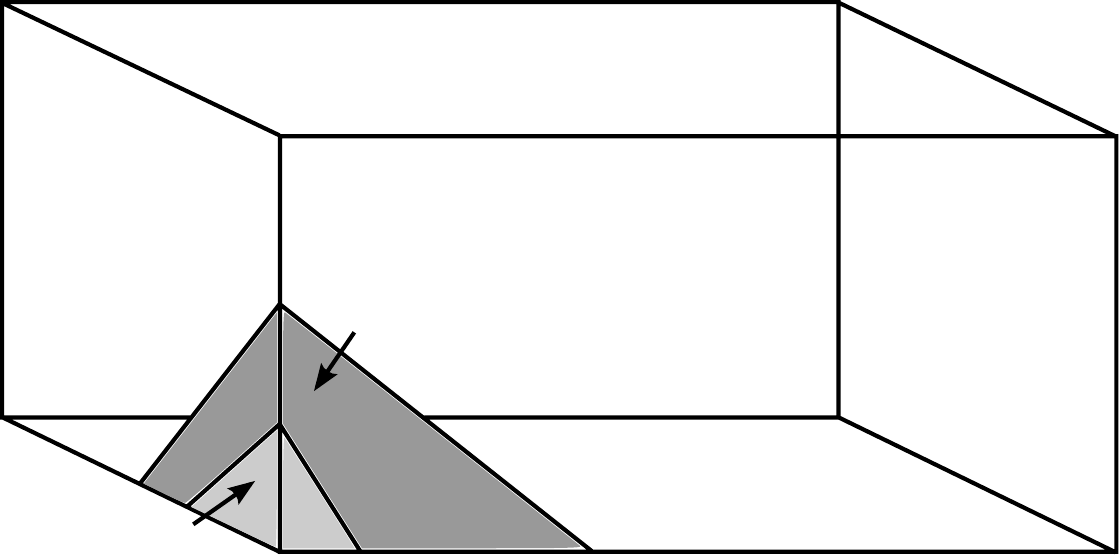}}%
    \put(0.25,0.42){\color[rgb]{0,0,0}\makebox(0,0)[lb]{\smash{$\nu_x=\delta_{U_{s}}$}}}%
    \put(0.025,0.00887148){\color[rgb]{0,0,0}\makebox(0,0)[lb]{\smash{$\nu_x=\delta_R$}}}%
    \put(0.27,0.215){\color[rgb]{0,0,0}\makebox(0,0)[lb]{\smash{$\nu_x=\lambda\delta_{U_{s}}+\left(1-\lambda\right)\delta_{QU_{l}}$}}}%
  \end{picture}%
\endgroup
	\caption{Depiction of a measure $\nu\in\mathcal{A}_{c}$ such that $I\left(\nu\right)<I\left(\delta_{U_{s}}\right)$. In the light shaded region $\nu_x=\delta_R$ for some $R\in\,SO\left(3\right)$ so that austenite has nucleated at a corner; in the dark shaded region $\nu_x=\lambda\delta_{U_{s}}+\left(1-\lambda\right)\delta_{QU_{l}}$ for some $Q\in\,SO\left(3\right)$ and $l\in\lbrace1,\ldots,6\rbrace$ such that the matrices $R$ and $\lambda U_{s}+\left(1-\lambda\right)QU_{l}$ are rank-one connected, i.e. $\nu_x$ corresponds to a simple laminate between $U_{s}$ and $QU_{l}$ there forming a compatible interface with $R$. Note that the normals to the interfaces between austenite and the simple laminate (habit plane) and between the simple laminate and the pure phase of $U_{s}$ (twinned-to-detwinned interface) are different.}
	\label{fig:min_corner}
\end{figure}

\section{Remarks and conclusions}
\label{sec:RandC}

For a general energy functional of the form \[\int_{\Omega}W\left(\nabla y\left(x\right)\right)\,dx,\] known necessary conditions for a map $y$ to be a local minimizer are that $W$ is quasiconvex at $\nabla y\left(x_0\right)$ for all $x_0$ in the interior - quasiconvexity in the interior (Meyers \cite{meyers65}) - and at the boundary (faces) of $\Omega$ - quasiconvexity at the boundary (Ball and Marsden \cite{ballmarsden}). Recently, Grabovsky and Mengesha~\cite{grabovsky2009} showed that, along with the satisfaction of the Euler-Lagrange equations and the positivity of the second variation, strengthened versions of the quasiconvexity conditions are in fact sufficient for $y$ to be a local minimizer; however, they showed this under smoothness assumptions on $W$ and also on the domain $\Omega$ which do not allow for edges or corners.

In our work, the condition that
\[I\left(\nu\right)\geq I\left(\delta_{U_{s}}\right)\quad\mbox{for all $\nu\in\mathcal{A}_{i}$ (resp. $\mathcal{A}_{f},\,\mathcal{A}_{e}$ and $\mathcal{A}_{c}$)}\]
is the appropriate expression of quasiconvexity at $U_{s}$ in the interior (resp. on faces, edges and corners). Then a way of interpreting Theorem~\ref{thm:main} is that $W$ is quasiconvex at $U_{s}$ in the interior, at the boundary (faces) and edges but not at corners, so that $U_s$ is a local minimizer in the interior, on faces and edges with respect to the localized variations defined before. We note that, to the best of the authors' knowledge, quasiconvexity conditions at edges and corners have not been considered before (see~\cite{ballkoumatos}).

The sets $\mathcal{M}_{s}$ and $\mathcal{M}^{-1}_{s}$ depend on the specific change of symmetry of the crystal lattice and, hence, on the lattice parameters of the material. For a range of parameters (see \cite{ballkoumatos} for details), including those of the specimen studied here, the above sets have explicit representations making our result applicable to a variety of parallelepipeds; for $s=1,2$ these are given by
\[\mathcal{M}_{s}=\lbrace e\in S^{2}:\left(-1\right)^{s-1}e_{2}e_{3}\geq0,\vert e_{1}\vert\leq\min\lbrace\vert e_{2}\vert ,\vert e_{3}\vert\rbrace\rbrace,\]
\[\mathcal{M}^{-1}_{s}=\lbrace e\in S^{2}:\left(-1\right)^{s-1}e_{2}e_{3}< 0,\vert e_{1}\vert >\max\lbrace\vert e_{2}\vert ,\vert e_{3}\vert\rbrace\rbrace\cup (1,0,0)^T\]
whereas for $s=3,4$ and $s=5,6$ we simply interchange $e_{1}$ with $e_{2}$ and $e_{3}$ respectively. In particular, our result applies to the Cu-Al-Ni specimen of this paper for any $s=1,\ldots\,,6$. However, for these lattice parameters, $\mathcal{M}_{s}\cup U^{-2}_{s}\mathcal{M}^{-1}_{s}$ does not exhaust the unit sphere. Hence our result leaves open the possibility that for different specimens nucleation could occur at a face or an edge.

It is worth noting that the same nucleation mechanism was observed for a Cu-Al-Ni specimen stabilized as a compound twin. This microstructure is also not able to form directly compatible interfaces with austenite and our methods may be applicable to this case as well.

Lastly, similar situations in which the incompatibility of gradients results in hysterisis have been documented before in different contexts, e.g.~\cite{ballchujames}. There, though in a different way, the mathematical analysis argues that despite the existence of a state with lower energy than a certain martensitic variant, it is necessarily geometrically incompatible with it, giving rise to an energy barrier which keeps the specific martensitic state stable. In general, in the context of microstructure formation, the incompatibility of gradients gives rise to very rich and interesting phenomena, such as the first genuinely non-classical austenite-martensite interfaces observed by Seiner and Landa~\cite{seinerlanda}, where austenite was able to form stress-free interfaces with a double laminate of martensite. In \cite{icomat08}, the reader can find further details as well as a relevant mathematical analysis.

\section*{Acknowledgement}

J.~M.~Ball and K.~Koumatos were supported by the EPSRC New Frontiers in the Mathematics of Solids (OxMOS) programme (EP/D048400/1) and the EPSRC award to the Oxford Centre for Nonlinear PDE (EP/E035027/1). H.~Seiner was supported by the Czech Science Foundation (project No.GAP107/10/0824) and the Institute of Thermomechanics ASCR v.v.i. (CEZ:AV0Z20760514).


\begin{thebibliography}{00}
\bibitem{ballchujames} J. M. Ball, C. Chu, R. D. James, Hysteresis during stress-induced variant rearrangement, J. de Physique. IV C8, 5 (1) (8) (1995) 245–251.
\bibitem{balljames87} J. M. Ball, R. D. James, Fine phase mixtures as minimizers of energy, Arch. Rational Mech. Anal. 100 (1) (1987) 13–52
\bibitem{balljames92} J. M. Ball, R. D. James, Proposed experimental tests of a theory of fine microstructure and the two-well problem, Phil. Trans. Roy. Soc. London A 338 (1650) (1992) 389-450.
\bibitem{ballmarsden} J.M. Ball, J.E. Marsden, Quasiconvexity at the boundary, positivity of the second variation and elastic stability, Arch. Rational Mech. Anal., 86 (3) (1984) 251-277.
\bibitem{ballkoumatos} J. M. Ball, K. Koumatos, \textit{in preparation}.
\bibitem{icomat08} J.M. Ball, K. Koumatos, H. Seiner, An analysis of non-classical austenite-martensite interfaces in CuAINi, Proceedings ICOMAT08, TMS (2010) 383-390 (arXiv:1108.6220).
\bibitem{bhattacharya} K. Bhattacharya, Microstructure of martensite: Why it forms and how it gives rise to the shape-memory effect, Oxford Series on Materials Modelling, Oxford University Press, Oxford, 2003.
\bibitem{Cernoch_JalCom}T. \v{C}ernoch, M. Landa, V. Nov\'{a}k, P. Sedl\'{a}k, P. \v{S}ittner, Acoustic characterization of the elastic properties of austenite phase and martensitic transformations in CuAlNi shape memory alloy, J.  Alloys Compounds 378 (2004) 140-144.
\bibitem{grabovsky2009} Y. Grabovsky, T. Mengesha, Sufficient conditions for strong local minima: the case of $C^1$ extremals, Trans. Amer. Math. Soc. 361 (3) (2009) 1495-1541.
\bibitem{meyers65} N.G. Meyers, Quasi-convexity and lower semi-continuity of multiple variational integrals of any order, Trans. Amer. Math. Soc., 119 (1965) 125-149.
\bibitem{cn} P. G. Ciarlet, J. Ne{\v{c}}as, Injectivity and self-contact in nonlinear elasticity, Arch. Rational Mech. Anal. 97 (3) (1987) 171-188.
\bibitem{J_AE} M. Landa, V. Nov\'{a}k, M. Blah\'{a}\v{c}ek, P. \v{S}ittner, Transformation processes in shape memory alloys based on monitoring acoustic emission activity, J. Acoust. Emission 20 (2002) 163-171.
\bibitem{ISPMA}{M. Landa, P. \v{S}ittner, V. Nov\'{a}k, P. Sedl\'{a}k, H. Seiner, \emph{Temperature dependence of elastic properties of cubic and orthorhombic phases in CuAlNi shape memory alloy near their stability limits}, Unpublished lecture, 10$^{\rm th}$ International Symposium on Physics of Materials, ISPMA-10, Prague (Czech Republic), August 30 - September 2, 2005.}
\bibitem{stab_PC} Y. Liu and D. Favier, Stabilisation of martensite due to shear deformation via variant reorientation in polycrystalline NiTi,  Acta Materialia, 48 (13) (2000) 3489-3499.
\bibitem{stab_SC} P. Picornell, V.A. Lvov, J. Pons, E. Cesari, Experimental and theoretical study of mechanical stabilization of martensite in Cu-Al-Ni single crystals, Mat. Sci. Eng. A. 438-440 (2006) 755-762.
\bibitem{seinerlanda} H. Seiner, M. Landa. Non-classical austenite-martensite interfaces observed in single crystals of Cu-Al-Ni, Phase Trans. 82 (2009), 793–807.
\bibitem{PrEASc}H. Seiner, M. Landa,  P. Sedl\'{a}k, Propagation of an Austenite-Martensite Interface in a Thermal Gradient, Proc. Estonian Acad. Sci. 56 (2007) 218-225.
\bibitem{PhaseTran} H. Seiner, P. Sedl\'{a}k, M. Landa, Shape recovery mechanism observed in single crystals of Cu-Al-Ni shape memory alloy, Phase Trans. 81 (2008) 537-551.
\end{thebibliography}
\end{document}